\documentclass[12pt,reqno]{amsart}

\def\usehyperref{1}
\usepackage{amssymb, amsmath, amsthm}
\usepackage{calligra, mathrsfs}
\usepackage{upgreek}
\usepackage{graphicx}
\usepackage{url}
\usepackage{mathtools}
\usepackage{enumerate}
\usepackage{verbatim}
\usepackage[retainorgcmds]{IEEEtrantools}
\usepackage{tikz-cd}
\usetikzlibrary{positioning}
\usepackage{microtype}

\usepackage[margin=1in,marginparwidth=0.8in, marginparsep=0.1in]{geometry}

\ifx\usehyperref\undefined
\usepackage{xr}
\newcommand\texorpdfstring[2]{#1}
\newcommand\nolinkurl[1]{\url{#1}}
\else
\usepackage{xr-hyper}
\usepackage[linktocpage=true, unicode]{hyperref}
\usepackage{xcolor}
\hypersetup{
    colorlinks,
    linkcolor={red!65!black},
    citecolor={blue!78!black}
    }
\fi

\newcommand{\Acal}{{\mathcal A}}

\newcommand{\Ccal}{{\mathcal C}}
\newcommand{\Dcal}{{\mathcal D}}

\newcommand{\Fcal}{{\mathcal F}}

\newcommand{\Mcal}{{\mathcal M}}

\newcommand{\Tcal}{{\mathcal T}}

\newcommand{\acal}{\Acal}
\newcommand{\ccal}{\Ccal}

\usepackage{scalerel,stackengine}

\makeatletter
\g@addto@macro\bfseries{\boldmath}
\makeatother

\newcommand{\kr}{\kern -2pt}

\DeclareMathOperator{\Hom}{Hom}
\DeclareMathOperator{\End}{End}

\DeclareMathOperator{\Maps}{Maps}

\newcommand{\enh}{{\text{\normalfont enh}}}

\newcommand{\pre}{{\normalfont \text{pre}}}
\newcommand{\co}{{\normalfont\text{co}}}
\newcommand{\enr}{{\normalfont\text{enr}}}
\newcommand{\st}{{\text{\normalfont st}}}

\newcommand{\closed}{{\normalfont \text{closed}}}

\DeclareMathOperator{\Cat}{Cat}

\DeclareMathOperator{\Tot}{Tot}

\let\Pr\relax
\DeclareMathOperator{\Pr}{Pr}

\DeclareMathOperator{\Mor}{Mor}

\DeclareMathOperator{\Mod}{Mod}

\DeclareMathOperator{\Algbrd}{Algbrd}

\newcommand{\B}{{\mathrm{B}}}

\DeclareMathOperator{\Spec}{Spec}

\DeclareMathOperator{\QCoh}{QCoh}
\DeclareMathOperator{\IndCoh}{IndCoh}

\DeclareMathOperator{\Corr}{Corr}
\DeclareMathOperator{\Sh}{Sh}

\DeclareMathOperator{\Cob}{Cob}
 
\newtheorem{proposition}[subsection]{Proposition}

\newtheorem{theorem}[subsection]{Theorem}
\newtheorem{corollary}[subsection]{Corollary}

\newtheoremstyle{note}{8.0pt plus 2.0pt minus 4.0pt}{8.0pt plus 2.0pt minus 4.0pt}{}{}{\bfseries}{.}{.5em}{} 
\theoremstyle{note}
\newtheorem{example}[subsection]{Example}
\newtheorem{remark}[subsection]{Remark}
\newtheorem{notation}[subsection]{Notation}
\newtheorem{construction}[subsection]{Construction}

\newtheorem{definition}[subsection]{Definition}

\AtBeginDocument{\def\MR#1{}}

\title{Categorification of sheaf theory}

\author{G. Stefanich}

\date{}

\begin{document}

 
\begin{abstract}
We discuss a systematic procedure for categorifying presentable six-functor formalisms. Our main result produces, given the input of a representation of the $\infty$-category of correspondences of an $\infty$-category with finite limits $\ccal$, a compatible sequence of representations of the $(\infty,n)$-category of correspondences of $\ccal$ for every $n \geq 1$. As an application, we explain a general recipe for constructing topological field theories.
\end{abstract}

\maketitle
  
\tableofcontents

 
\section{Introduction}

Let $X$ be an algebraic stack over a base ring $k$, and let $\QCoh(X)$ be the stable $\infty$-category of quasi-coherent sheaves on $X$. In good cases, for each closed manifold $M$ we have an identification
\begin{align*}
\int_M \QCoh(X) = \QCoh(\operatorname{Maps}(M, X)) 
\end{align*}
where the left hand side denotes the factorization homology of $\QCoh(X)$ regarded as an $E_\infty$-algebra in presentable stable $k$-linear $\infty$-categories, and $\operatorname{Maps}(M, X)$ denotes the derived moduli stack parametrizing locally constant maps from $M$ into $X$.

The fundamental ingredient needed in order to make the above factorization homology computation is the categorical K\"{u}nneth formula: in good cases, for every pair of maps of algebraic stacks $Y \rightarrow S$ and $Z \rightarrow S$ one has an identification
\begin{align*}
\QCoh(Y) \otimes_{\QCoh(S)} \QCoh(Z) =  \QCoh(Y \times_S Z).
\end{align*}
This is the case for instance when all the stacks involved are perfect in the sense of \cite{BZNF}; this includes schemes, and many algebraic stacks in characteristic zero.

When working with more general sheaf theories, the categorical K\"{u}nneth formula often fails. This is the case for instance in the theories of ind-coherent sheaves and \'{e}tale sheaves. Even for quasi-coherent sheaves, when working in positive characteristic   the categorical K\"{u}nneth formula is known only under fairly restrictive conditions, and we expect that it in fact often fails, given the known pathologies for $\QCoh$ in that setting \cite{StefanichDualizability}. In light of this, it is not surprising that the more general formula
\begin{align*}
\int_M \Sh(X) = \Sh(\operatorname{Maps}(M, X)) 
\end{align*}
often fails. Examples are not hard to find: if $\Sh = \IndCoh$ then the formula is false already when $M$ is a circle and $X$ is the affine line.

The goal of this note is to present a framework that, among other things, allows one to obtain a replacement of the factorization homology formula for a general sheaf theory $\Sh$. The general statement will not (and cannot) be about factorization homology; instead, we will obtain $\Sh(\operatorname{Maps}(M, X))$ as the value on $M$ of a topological field theory. 


\subsection*{Sheaf-theoretic topological field theories}\label{subsection field theoretic}

The connection between factorization homology and topological field theory is given by \cite{Scheimbauer}: factorization homology provides the values of topological field theories valued in Morita higher categories. Thus one can recast the original $\QCoh$-factorization homology formula as follows: for good enough $X$, there is an equivalence
\[
\chi_{n\kr\Mor, \QCoh, X}(M) = \QCoh(\Maps(M, X))
\]
where 
\[
\chi_{n\kr\Mor, \QCoh, X}: n\kr\operatorname{Cob} \rightarrow n\kr\operatorname{Mor}(\Pr_{\st, k})
\]
is a fully extended unoriented topological field theory valued in the Morita theory of $E_n$-monoidal presentable stable $k$-linear $\infty$-categories, with $n$ being the dimension of $M$. We remark that although the source of $\chi_{n\kr\Mor, \QCoh, X}$ involves cobordisms of dimension at most $n$, we may secretly think of this as being an $(n+2)$-dimensional topological field theory, given the categorical complexity of its outputs: indeed, its values on manifolds of dimension $n$ are not numbers but $\infty$-categories.

In the same way that the factorization homology formula does not work for general sheaf theories, we should not expect to obtain $\Sh(\Maps(M, X))$ as the value of a topological field theory valued in $E_n$-monoidal presentable $\infty$-categories. Instead, our topological field theories will be valued in presentable higher categories \cite{Pres}. 

Up to size issues, one may think about a presentable $(\infty,n)$-category inductively as being an $(\infty,n)$-category which has colimits and whose Homs are presentable $(\infty, n-1)$-categories.\footnote{Note however that this is not \emph{literally} true: it turns out that Homs are not always presentable when $n > 2$, see \cite{Aoki}. Officially, one first considers the $(\infty,n+1)$-category of  $(\infty,n)$-categories admitting small colimits for cells of dimension $d$ for all $0 \leq d < n$, and then singles out a full subcategory of presentable objects inside there. We refer to \cite{Pres} for details.} The totality of all presentable $(\infty,n)$-categories forms a symmetric monoidal $(\infty, n+1)$-category $n\kr\Pr$. Imposing stability and $k$-linearity at the level of $(n-1)$-cells yields a variant $n\kr\Pr_{\st, k}$ which receives a symmetric monoidal functor
\[
n\kr\Mod: n\kr\Mor(\Pr_{\st, k}) \rightarrow (n+1)\kr\Pr_{\st,k}.
\]
Already for $\QCoh$, the passage from  $n\kr\Mor(\Pr_{\st, k})$ to $(n+1)\kr\Pr_{\st,k}$ has the benefit of allowing one to formulate a result which places no restrictions on the stacks:

\begin{theorem}\label{theorem qcoh}
For every algebraic stack $X$ over $k$ and every $n \geq 0$, there is a symmetric monoidal functor $\chi_{(n+1)\kr\QCoh, X}: n\kr\Cob \rightarrow (n+1)\kr\Pr_{\st, k}$ 
such that for every closed manifold $M$ of dimension $n$ we have an equivalence
\[
\chi_{(n+1)\kr\QCoh, X}(M) = \QCoh(\Maps(M, X)).
\]
\end{theorem}

The value of $\chi_{(n+1)\kr\QCoh, X}$ on the point is given by the $(\infty,n+1)$-category $(n+1)\kr\QCoh(X)$ of quasi-coherent sheaves of $(\infty, n)$-categories on $X$ which we introduced in \cite{Thesis}. The above theorem is a direct consequence of the functoriality properties of $(n+1)\kr\QCoh$. More precisely, it follows from the fact that $(n+1)\kr\QCoh$ gives rise to a functor out of a higher category of correspondences \cite{Thesis}, combined with a general construction of topological field theories valued in correspondences, \cite{CalaqueHaugsengScheimbauer} theorem D (see also \cite{BZNtraces} for a version of this theme in the context of one dimensional theories). From this perspective, the role of the categorical K\"{u}nneth formulas in factorization homology computations is played by the base change formulas in the theory of sheaves of higher categories.

The connection between $\chi_{(n+1)\kr\QCoh, X}$ and $\chi_{n\kr\Mor, \QCoh, X}$ is given by the affineness theorems of \cite{G, Thesis}:  for $X$ good enough, one has an equivalence 
\[
(n+1)\kr\QCoh(X) = n\kr\Mod_{\QCoh(X)}.
\]
As shown in \cite{StefanichDualizability}, $1$-affineness often fails beyond characteristic zero, so although theorem \ref{theorem qcoh} still holds, it cannot be rephrased in Morita theory terms.

The output of this note is a theory that allows one to generalize theorem \ref{theorem qcoh} beyond the quasi-coherent setting. In fact, there is nothing special about the setting of algebraic stacks: one may work on a general $\infty$-category with finite limits:

\begin{theorem}\label{theorem tqft sh}
Let $\Ccal$ be an $\infty$-category with finite limits and let $\Sh$ be a presentable six-functor formalism on $\Ccal$ (see section \ref{subsection categorify} below). Then there exists a sequence of presentable symmetric monoidal $(\infty, n+1)$-categories $n\kr\Pr_{\Sh}$ with $\End_{n\kr\Pr_{\Sh}}(1_{n\kr\Pr_{\Sh}}) = (n-1)\kr\Pr_{\Sh}$ for all $n \geq 2$, with the following feature: for every object $X$ in $\Ccal$ and every $n \geq 0$ there is a symmetric monoidal functor $
\chi_{(n+1)\kr\Sh, X} : (n+1)\kr\Cob \rightarrow (n+1)\kr\Pr_{\Sh}$
such that for each closed manifold $M$ of dimension $n$ there is an equivalence 
\[
\Gamma( \chi_{(n+1)\kr\Sh, X}(M) ) = \Sh(\Maps(M, X))
\]
where $\Gamma(-) = \Hom_{1\kr\Pr_{\Sh}}(1_{\kr\Pr_{\Sh}, -})$ denotes the passage to the underlying $\infty$-category.
\end{theorem}

Theorem \ref{theorem tqft sh} works formally just like theorem \ref{theorem qcoh} and is a direct consequence of the existence of a family of categorifications $n\kr\Sh$ of $\Sh$ with good  functoriality properties, which is what this note aims to establish. The only extra twist beyond the quasi-coherent setting is that the targets $n\kr\Pr_{\Sh}$ may be somewhat exotic. One in fact has that $n\kr\Pr_{\Sh} = (n+1)\kr\Sh(1_\ccal)$, and  for each closed manifold $M$ of dimension $0 \leq d \leq n$, the value of $\chi_{(n+1)\kr\Sh, X}$ on $M$ is given by a canonical upgrade of $(n+1-d)\kr\Sh(X)$ to $(n+2-d)\kr\Sh(1_{\Ccal})$.

The topological field theories arising from theorem \ref{theorem tqft sh} are particularly interesting in the case $\Sh = \IndCoh$. Specializing to dimension three, one obtains the following:

\begin{corollary}\label{corollary rw}
Let $X$ be a smooth scheme over a field $k$ of characteristic zero. Then there exists a symmetric monoidal functor $\chi_{2\kr\IndCoh, X}: 2\kr\Cob \rightarrow 2\kr\Pr_{\IndCoh}$ such that
\[
\chi_{2\kr\IndCoh, X}(S^1) = \IndCoh(\Maps(S^1, X)) = \QCoh(T^*[2]X).
\]
\end{corollary}

The topological field theory from corollary \ref{corollary rw} may be thought of as a mathematical incarnation of the Rozansky-Witten theory of $T^*X$ \cite{KRS}. More precisely, $\chi_{2\kr\IndCoh, X}$ satisfies the design criteria for a conical version of Rozansky-Witten theory; non-conical versions can also be obtained after a suitable two-periodization of this construction. As one may expect, $2\kr\IndCoh$ enjoys symmetries arising from symmetries of $T^*X$; the most fundamental of these is a $2$-categorical Fourier transform, which is the subject of ongoing work joint with David Ben-Zvi and David Nadler.

The $\infty$-category $\IndCoh(\Maps(S^1, X))$ arises not only as the value on $S^1$ of the theory, but also as the center: $\IndCoh(\Maps(S^1, X))$ has a canonical acion on $2\kr\IndCoh(X)$. Via Koszul duality, this allows one to express $2\kr\IndCoh(X)$ as the global sections of a sheaf of $(\infty,2)$-categories on $T^*X$ with its conical Zariski topology. In other words, the theory of ind-coherent sheaves of categories on $X$ is not only local on $X$, but also admits a microlocal theory, which may be thought of as a categorification of the microlocal theory of ind-coherent sheaves developed in \cite{AG}.

The theory of ind-coherent sheaves of categories lacks some features present in the classical theory of $\IndCoh$. Fundamentally, one could hope for a notion of coherent sheaves of categories, including inside $2\QCoh(X)$ and which recovers $2\kr\IndCoh(X)$ via a suitable completion procedure. In work in progress with Carlos di Fiore, we show that a version of this hope can be realized if one works in the context of sheaves of small idempotent complete stable categories. The resulting microlocal picture very closely matches the one from the theory of D-modules: the singular support measures the difference between perfect and coherent sheaves of categories, and between star and shriek pullbacks, in what can be thought of as a categorification of the theory of vanishing cycles.

In four dimensions, theorem \ref{theorem tqft sh} may be specialized to yield the following:

\begin{corollary}\label{corollary langlands}
Let $G$ be an affine algebraic group over a field $k$ of characteristic zero. Then there exists a  symmetric monoidal functor $\chi_{3\kr\IndCoh, \B G} : 3\kr\Cob \rightarrow 3\kr\Pr_{\IndCoh}$ such that for every closed manifold $M$ of dimension $2$ we have an equivalence 
\[
\chi_{3\kr\IndCoh, \B G} (M) = \IndCoh(\Maps(M, X)).
\]
\end{corollary}

Corollary \ref{corollary langlands} is consistent with the expectation that the various structures arising in the Betti form of the geometric Langlands program may be organized into the framework of $4$-dimensional topological field theory  \cite{KapustinWitten, BZN}. More precisely, it provides a realization of a version of the spectral Betti geometric Langlands TQFT with no restriction on the central parameters (that is, no conditions of nilpotent singular support).


\subsection*{Categorification of six-functor formalisms}\label{subsection categorify}

As remarked above, theorem \ref{theorem tqft sh} is really a consequence of a general result on the existence of categorifications of a sheaf theory, which we proceed to explain. Let $\ccal$ be an $\infty$-category with finite limits. Associated to $\ccal$ there is a symmetric monoidal $\infty$-category $\Corr(\ccal)$ called the $\infty$-category of correspondences, with the following features:
\begin{itemize}
\item The anima of objects of $\Corr(\ccal)$ agrees with the anima of objects of $\ccal$.
\item Let $X$ and $Y$ be a pair of objects of $\ccal$. Then a morphism from $X$ to $Y$ in $\Corr(\ccal)$ consists of a span $X \leftarrow S \rightarrow Y$ in $\ccal$.
\item Let $X \leftarrow S \rightarrow Y$ and $Y \leftarrow T \rightarrow Z$ be a pair of spans in $\ccal$, which we regard as morphisms in $\Corr(\ccal)$. Then their composition is given by the span $X \leftarrow S \times_Y T \rightarrow Z$.
\item Let $X$ and $Y$ be a pair of objects of $\ccal$. Then their tensor product in $\Corr(\ccal)$ is given by $X \times Y$.
\end{itemize}

As explained in \cite{GR}, the $\infty$-category $\Corr(\ccal)$ (and variants of it) provides a convenient way of capturing the functoriality present in various sheaf theories of interest. More precisely, if $\Tcal$ is a symmetric monoidal $(\infty,2)$-category whose objects we think about as being $\infty$-categories of some sort, then a lax symmetric monoidal functor $\Sh: \Corr(\ccal) \rightarrow \Tcal$ gives rise to the following data:
\begin{itemize}
\item For each object $X$ in $\ccal$ an object $\Sh(X)$ in $\Tcal$ which we think of as the $\infty$-category of sheaves on $X$.
\item For each map $f: X \rightarrow Y$ in $\ccal$ a pair of morphisms $f^*: \Sh(Y) \rightarrow \Sh(X)$ and $f_! : \Sh(X) \rightarrow \Sh(Y)$ which we think of as pullback and pushforward functors.
\item For each cartesian square
\[
\begin{tikzcd}
X' \arrow{d}{f'} \arrow{r}{g'} & X \arrow{d}{f} \\
Y' \arrow{r}{g} & Y
\end{tikzcd}
\]
in $\ccal$, an isomorphism $f'_! g'^* = g^*f_!$.
\item For each pair of objects $X, Y$ in $\ccal$, a morphism $\boxtimes: \Sh(X) \otimes \Sh(Y) \rightarrow \Sh(X \times Y)$ which we think about as the exterior tensor product functor.
\end{itemize}

Lax symmetric monoidal functors out of $\Corr(\Ccal)$ are called three functor formalisms, and requiring right adjoints to $f_!, f^*$ and $\Fcal \boxtimes -$ one arrives at the notion of a  six-functor formalism \cite{Mann, ScholzeSix}. The existence of these extra adjoints is automatic when $\Sh$ takes values in presentable categories and all the functors preserve colimits. In this case, a lax symmetric monoidal functor $\Sh: \Corr(\Ccal) \rightarrow \Pr$ is called a presentable six-functor formalism.

As explained in \cite{HSTI, Thesis}, when working with theories of sheaves of higher categories one tends to encounter much stronger functoriality properties than in ordinary sheaf theory. In this context, the role of $\Corr(\ccal)$ is played by the symmetric monoidal $(\infty,n)$-category $n\kr\Corr(\ccal)$, which in the case $n = 1$ agrees with $\Corr(\ccal)$ and for $n > 1$ admits the following informal inductive description:
\begin{itemize}
\item  The anima of objects of $n\kr\Corr(\ccal)$ agrees with the anima of objects of $\ccal$.
\item Let $X$ and $Y$ be a pair of objects of $\ccal$. Then the $(\infty,n-1)$-category of morphisms from $X$ to $Y$ in $n\kr\Corr(\ccal)$ is given by $(n-1)\kr\Corr({}_{X\backslash}\ccal_{/Y})$.
\item Let $X$ and $Y$ be a pair of objects of $\ccal$.  Then their tensor product in $n\kr\Corr(\ccal)$ is given by $X \times Y$.
\end{itemize}

Our main result constructs, for every presentable six-functor formalism $\Sh: \Corr(\ccal) \rightarrow \Pr$, a compatible sequence of categorified formalisms $n\kr\Sh: n\kr\Corr(\ccal) \rightarrow n\kr\Pr$ which we think about as providing theories of sheaves of higher  categories of flavor $\Sh$. More precisely, we have the following:

\begin{theorem}\label{theo exists categorification}
Let $\ccal$ be an $\infty$-category with finite limits, and let $\Sh: \Corr(\ccal) \rightarrow \Pr$ be a lax symmetric monoidal functor. Then there exists:
\begin{itemize}
\item A sequence of presentable symmetric monoidal $(\infty,n+1)$-categories $n\kr\Pr_{\Sh}$ with  $\End_{n\kr\Pr_{\Sh}}(1_{n\kr\Pr_{\Sh}}) = (n-1)\kr\Pr_{\Sh}$ for all $n \geq 2$.
\item  A sequence of symmetric monoidal functors $n\kr\Sh^\sharp : n\kr\Corr(\ccal) \rightarrow n\kr\Pr_{\Sh}$ such that the resulting square
\[
\begin{tikzcd}
\End_{n\kr\Corr(\ccal)}(1_\ccal) \arrow{r}{n\kr\Sh^\sharp} \arrow{d}{=} & \End_{n\kr\Pr_{\Sh}}(1_{n\kr\Pr_{\Sh}}) \arrow{d}{=} \\
(n-1)\kr\Corr(\ccal) \arrow{r}{(n-1)\kr\Sh^\sharp} & (n-1)\kr\Pr_{\Sh}
\end{tikzcd}
\]
commutes, and with the feature that $\Sh(-) = \Hom_{1\kr\Pr_{\Sh}}(1_{1\kr\Pr_{\Sh}}, 1\kr\Sh^\sharp(-))$.
\end{itemize}
The above are furthermore subject to the following properties:
\begin{itemize}
\item Given $n \geq 2$ and an object $X$ in $\Ccal$, the $(\infty,n)$-category of maps $1_{n\kr\Pr_{\Sh}} \rightarrow n\kr\Sh(X)^\sharp$ is generated under colimits by the image under $n\kr\Sh^\sharp$ of the morphisms $1_{n\kr\Corr(\Ccal)} \rightarrow X$ in $n\kr\Corr(\Ccal)$.
\item The unique symmetric monoidal functor $n\kr\Pr \rightarrow n\kr\Pr_{\Sh}$ admits a colimit preserving right adjoint $\Gamma(-)$.
\end{itemize} 
\end{theorem}	

In the language of \cite{Thesis}, the sequence of higher presentable categories $n\kr\Pr_{\Sh}$ assembles into a categorical spectrum, and the sequence of functors $n\kr\Sh^\sharp$ assembles into a representation of the categorical spectrum of correspondences of $\ccal$.

Although in general the targets $n\kr\Pr_{\Sh}$ are different from $n\kr\Pr$, one may obtain (non symmetric monoidal) functors valued in $n\kr\Pr$ by setting $n\kr\Sh(-) = \Gamma(n\kr\Sh(-)^\sharp)$. In the case $n = 1$ this recovers the starting sheaf theory. In general, the conditions from the statement of theorem \ref{theo exists categorification} imply the following inductive description:
\begin{itemize}
\item $(n+1)\kr\Sh(X)$ is a presentable $(\infty,n+1)$-category freely generated under weighted colimits by the objects $n\kr\Sh(Y/X)$ which are obtained by pushforward along a map $Y \rightarrow X$ of the unit of $(n+1)\kr\Sh(Y)$.
\item Given a pair of maps $Y \rightarrow X$ and $Z \rightarrow X$ one has 
\[
\Hom(n\kr\Sh(Y/X), n\kr\Sh(Z/X)) = n\kr\Sh(Y\times_X Z).
\]
\end{itemize}

In the case when $\Sh = \IndCoh$, a construction of $2\kr\Sh(X)$ had been proposed by Arinkin and Gaitsgory in the case of smooth schemes; meanwhile, in the setting of \'etale sheaves, the $(\infty,2)$-category $2\kr\Sh(\Spec(k))$ is studied by Gaitsgory, Rozenblyum and Varshavsky, following an earlier proposal by Drinfeld, see \cite{GaitsgoryICM} sections 2.2 and 2.6. In general, at the level of generators the above description recovers, in the case $n = 1$, the so-called $(\infty,2)$-category of kernels of $\Sh$, which has become in recent years a fundamental tool in the study of six-functor formalisms \cite{LuZheng, FarguesScholze, HansenScholze, ScholzeSix, Zavyalov, HeyerMann}. 

The $(\infty, n)$-category $n\kr\Sh(X)$ is a priori rather complex: the above description shows that it has generators which are parametrized by maps $Y \rightarrow X$. It is however often possible to show that a smaller collection of generators suffice. We refer to section \ref{section compute} for a discussion of the basic computational toolset that allows this.


\subsection{Acknowledgments}

I am grateful to Ko Aoki, Dima Arinkin, David Ben-Zvi, Carlos di Fiore, Dennis Gaitsgory, David Nadler, Sam Raskin, Juan Esteban Rodr\'iguez Camargo, Peter Scholze, and Jackson Van-Dyke for conversations connected to the topic of this paper. Part of this work was carried out at the Max Planck Institute for Mathematics in Bonn, and I am grateful to the institute for its hospitality and support.


\subsection{Conventions and notation}\label{section conventions}

In the remainder of this text we follow the convention where the word category stands for $\infty$-category, and use the term $n$-category to refer to $(\infty,n)$-categories. We also use the terms scheme and stack to refer to spectral schemes and stacks.

We make use throughout this text of the theory of enriched (higher) categories. For each possibly large monoidal category $\Mcal$ we denote by $\Cat^\Mcal$ the category of $\Mcal$-enriched categories with a small set of isomorphism classes of objects. In the case when $\Mcal$ is symmetric monoidal, $\Cat^\Mcal$ has an induced symmetric monoidal structure, so it makes sense to define $n\kr\Cat^\Mcal$ inductively for all $n \geq 1$ as follows:
\begin{itemize}
\item If $n = 1$ then $n\kr\Cat^\Mcal = \Cat^\Mcal$.
\item If $n > 1$ then $n \kr\Cat^\Mcal = \Cat^{(n-1)\kr\Cat^\Mcal} = (n-1)\kr\Cat^{\Cat^\Mcal}$.
\end{itemize} 
Objects of $n\kr\Cat^\Mcal$ are called $\Mcal$-enriched $n$-categories. In the special case when $\Mcal$ is the category of small anima we set $n\kr\Cat = n\kr\Cat^{\Mcal}$, and call it the category of $n$-categories.

The assignment $\Mcal \mapsto \Cat^\Mcal$ is functorial in lax monoidal functors in $\Mcal$. Given such a functor $F: \Mcal \rightarrow \Mcal'$ we denote by $F_!: \Cat^\Mcal \rightarrow \Cat^{\Mcal'}$ the induced map. In the case when $F$ is a (lax) symmetric monoidal functor between symmetric monoidal categories we have that $F_!$ is also (lax) symmetric monoidal.

If $A$ is an algebra in a monoidal category $\Mcal$ we denote by $BA$ the corresponding $\Mcal$-enriched category. The assignment $A \mapsto BA$ provides a fully faithful embedding from the category of algebras in $\Mcal$ into the category of pointed $\Mcal$-enriched categories. We denote by $\Omega$ its right adjoint; in other words, this is the functor which sends a pointed $\Mcal$-enriched category $(\Dcal, X_0)$ to the algebra of endomorphisms of $X_0$.

In the case when $\Mcal$ is symmetric monoidal, the assignment $A \mapsto BA$ has a symmetric monoidal structure. In particular, if $A$ is a commutative algebra then $BA$ has the structure of symmetric monoidal $\Mcal$-enriched category. Iterating the construction $A \mapsto BA$ we may thus define for each $n \geq 1$ a symmetric monoidal $\Mcal$-enriched $n$-category $B^nA$. The assignment $A \mapsto B^nA$ has a right adjoint which we denote by $\Omega^n$.

 
\section{Enrichment of \texorpdfstring{$n\kr\Corr(\ccal)$}{nCorr(C)}}

Let $\ccal$ be a category with finite limits. Then every object of $\Corr(\ccal)$ is self dual. In particular, the symmetric monoidal structure on $\Corr(\ccal)$ is closed, so that $\Corr(\ccal)$ has a canonical structure of $\Corr(\ccal)$-enriched category.  The goal of this section is to show that, more generally, $n\kr\Corr(\ccal)$ may be given the structure of a $\Corr(\ccal)$-enriched $n$-category for all $n \geq 1$.

\begin{notation}\label{notation gamma}
Let $\Mcal$ be a presentable symmetric monoidal category, and let $A$ be a commutative algebra in $\Mcal$. Let $G: \Mod_A(\Mcal) \rightarrow \Mcal$ be the forgetful functor, which we equip with its canonical lax symmetric monoidal structure. Consider the induced lax symmetric monoidal functor
\[
G_!: \Cat^{\Mod_A(\Mcal)} \rightarrow \Cat^\Mcal.
\]
This induces a lax symmetric monoidal functor
\[
\Cat^{\Mod_A(\Mcal)} = \Mod_{1_{\Cat^{\Mod_A(\Mcal)}}}(\Cat^{\Mod_A(\Mcal)}) \rightarrow \Mod_{G_! 1_{\Cat^{\Mod_A(\Mcal)}}}(\Cat^\Mcal).
\]
Observe that $G_! 1_{1_{\Cat^{\Mod_A(\Mcal)}}} = BA$ as symmetric monoidal $\Mcal$-enriched categories. We denote by 
\[
\Gamma_A : \Cat^{\Mod_A(\Mcal)} \rightarrow \Mod_{BA}(\Cat^\Mcal)
\]
the induced lax symmetric monoidal functor.
\end{notation}

\begin{proposition}\label{prop gamma equivalence}
Let $\Mcal$ be a presentable symmetric monoidal category and let $A$ be a commutative algebra in $\Mcal$. Then the lax symmetric monoidal functor $\Gamma_A$ from notation \ref{notation gamma} is a symmetric monoidal equivalence. 
\end{proposition}
\begin{proof}
Let $\Algbrd(\Mcal)$ (resp. $\Algbrd(\Mod_A(\Mcal))$) be the category of $\Mcal$-algebroids (resp. $\Mod_A(\Mcal)$-algebroids); that is, these are the categories of non-necessarily univalent enriched categories. Let $B^{\pre}A$ be the $\Mcal$-algebroid with a single object and endomorphisms $A$, and note that we have a lax symmetric monoidal functor
\[
\Gamma_A^\pre: \Algbrd(\Mod_A(\Mcal)) \rightarrow \Mod_{B^{\pre}(A)}(\Algbrd(\Mcal))
\]
defined similarly to $\Gamma_A$. The proposition will follow if we show that $\Gamma_A^\pre$ is a symmetric monoidal equivalence. Let $G_!^\pre: \Algbrd(\Mod_A(\Mcal)) \rightarrow \Algbrd(\Mcal)$ be the lax symmetric monoidal functor induced by $G$. Note that $G_!^\pre$ is right adjoint to the symmetric monoidal functor $F_!^\pre$ induced from the free functor $F = - \otimes A$. A monadicity argument reduces us to proving the following:
\begin{enumerate}[\normalfont (i)]
\item $G^\pre_!$ commutes strictly with the action of $\Algbrd(\Mcal)$.
\item Let $S_\bullet$ be a simplicial diagram in $\Algbrd(\Mod_A(\Mcal))$ and assume that $G^\pre_!S_\bullet$ is the simplicial Bar resolution of some $B^\pre A$-module in $\Algbrd(\Mcal)$. Then $G^\pre_!$ preserves the geometric realization of $S_\bullet$. 
\item $G^\pre_!$ is conservative.
\end{enumerate}
Item (iii) follows readily from the fact that $G$ is conservative, while item (i) follows from the fact that $G$ commutes strictly with the action of $\Mcal$. It remains to establish (ii). Since the anima of objects of $B^\pre A$ is a singleton the simplicial anima underlying $G^\pre_! S_\bullet$ is constant. It follows that the simplicial anima underlying $S_\bullet$ is constant. Consequently, we may reduce to showing that for each anima $J$ the functor $\Algbrd_J(\Mod_A(\Mcal)) \rightarrow \Algbrd_J(\Mcal)$ induced by $G$ on algebroids with anima of objects $J$ preserves geometric realizations. This follows from the fact that $G$ itself preserves geometric realizations.
\end{proof}

\begin{notation}\label{notation gamman}
Let $n \geq 1$. Let $\Mcal$ be a presentable symmetric monoidal category, and let $A$ be a commutative algebra in $\Mcal$.  We define a lax symmetric monoidal functor
\[
\Gamma_A^n: n\kr\Cat^{\Mod_A(\Mcal)} \rightarrow \Mod_{B^nA}(n\kr\Cat^\Mcal)
\]
by induction on $n$ as follows:
\begin{itemize}
\item If $n = 1$ we let $\Gamma_A^n$ be the lax symmetric monoidal functor $\Gamma_A$ from notation \ref{notation gamma}.
\item Assume that $n > 1$. Then we let $\Gamma_A^n$ be given by the composition 
\begin{align*}
n\kr\Cat^{\Mod_A(\Mcal)} = \Cat^{(n-1)\kr\Cat^{\Mod_A(\Mcal)}} &\xrightarrow{(\Gamma_A^{n-1})_!} \Cat^{\Mod_{B^{n-1}A}((n-1)\kr\Cat^\Mcal)}  \\ &\xrightarrow{\Gamma_{B^{n-1}A}}  \Mod_{B(B^{n-1}A)}(\Cat^{ (n-1)\kr\Cat^\Mcal}) \\ &=   \Mod_{B^nA}(n\kr\Cat^\Mcal) 
\end{align*}
\end{itemize}
\end{notation}

\begin{corollary}\label{coro gamman equiv}
Let $n \geq 1$. Let $\Mcal$ be a presentable symmetric monoidal category, and let $A$ be a commutative algebra in $\Mcal$.  Then the lax symmetric monoidal functor $\Gamma_A^n$ from notation \ref{notation gamman} is a symmetric monoidal equivalence.
\end{corollary}
\begin{proof}
Follows from an inductive application  of proposition \ref{prop gamma equivalence}.
\end{proof}

\begin{definition}
Let $\Acal$ be a symmetric monoidal category. We say that an $\Acal$-module category $\Dcal$ is closed if it admits all Hom objects. We denote by $\Mod_{\Acal}(\Cat)_{\closed}$ the full subcategory of $\Mod_\Acal(\Cat)$ on the closed $\Acal$-modules.  
\end{definition}

\begin{remark}\label{remark tensor vs tensor}
Let $\Acal$ be a symmetric monoidal category, and assume that every object of $\acal$ is dualizable. Then $\Mod_{\Acal}(\Cat)_{\closed}$ is closed under tensor products inside $\Mod_\acal(\Cat)$. Furthermore, the procedure of $\acal$-enrichment of closed $\Acal$-modules assembles into a fully faithful symmetric monoidal functor $\Mod_{\Acal}(\Cat)_{\closed} \rightarrow \Cat^\Acal$. 
\end{remark}

\begin{definition}\label{definition closed}
Let $n \geq 1$ and let $\Acal$ be a symmetric monoidal category, which we regard as a commutative algebra in $\Mcal = \Cat$. Assume that every object of $\acal$ is dualizable. We say that a $B^{n-1}\Acal$-module $\Dcal$ in $n\kr\Cat$ is closed if the inverse image of $\Dcal$ under $\Gamma^{n-1}_{\acal}$ belongs to $(n-1)\kr\Cat^{\Mod_{\Acal}(\Cat)_{\closed}} \subseteq (n-1)\kr\Cat^{\Mod_{\acal}(\Cat)}$.
\end{definition}

Definition \ref{definition closed} provides a way of equipping an $n$-category $\Dcal$ with an enrichment over a symmetric monoidal category $\Acal$ such that all objects of $\Acal$ dualizable: it is enough to equip $\Dcal$ with the structure of a closed module over $B^{n-1} A$. We now apply this to our case of interest:

\begin{proposition}\label{prop is closed}
Let $\ccal$ be a category with finite limits and let $n \geq 2$. Consider the symmetric monoidal functor
\[
B^{n-1}\Corr(\ccal) = B^{n-1}\Omega^{n-1} n\kr\Corr(\ccal) \rightarrow n\kr\Corr(\ccal)
\]
obtained from the counit of the $B^{n-1} \dashv \Omega^{n-1}$ adjunction. Then the induced $B^{n-1}\Corr(\ccal)$-module structure on $n\kr\Corr(\ccal)$ is closed. 
\end{proposition}

The closure conditions needed to establish proposition \ref{prop is closed} all follow from the following:

 \begin{proposition}\label{proposition internal hom}
 Let $\ccal$ be a category with finite limits. Let $X$ be an object of $\ccal$ and consider the canonical action of $\Corr(\ccal)$ on $\Corr(\Ccal_{/X})$. Then $\Corr(\Ccal_{/X})$ is a closed module over $\Corr(\Ccal)$.
 \end{proposition}
 \begin{proof}
 The module structure arises by restriction of scalars from the symmetric monoidal functor $\alpha: \Corr(\ccal) \rightarrow \Corr(\Ccal_{/X})$ induced by the functor $\ccal \rightarrow \Ccal_{/X}$ of product with $X$. Since $\alpha$ admits a right adjoint, we may reduce to showing that $\Corr(\ccal_{/X})$ is closed as a module over itself. Indeed, this follows from the fact that every object in a category of correspondences is dualizable.
 \end{proof}

\begin{notation}
Let $\ccal$ be a category with finite limits and let $n \geq 2$. It follows from proposition \ref{prop is closed} that the inverse image of $n\kr\Corr(\ccal)$ under $\Gamma^{n-1}_{\Corr(\ccal)}$ belongs to $(n-1)\kr\Cat^{\Mod_{\Corr(\ccal)}(\Cat)_{\closed}}$. We let $n\kr\Corr^\enr(\Ccal)$ be its image under the functor
\[
(n-1)\kr\Cat^{\Mod_{\Corr(\ccal)}(\Cat)_{\closed}} \rightarrow (n-1)\kr\Cat^{\Cat^{\Corr(\ccal)}} = n\kr\Cat^{\Corr(\ccal)}
\]
induced from the symmetric monoidal inclusion $\Mod_{\Corr(\ccal)}(\Cat)_{\closed} \rightarrow \Cat^{\Corr(\ccal)}$ (see remark \ref{remark tensor vs tensor}). We extend this notation also to the case $n = 1$ by letting $\Corr^\enr(\ccal)$ be the $\Corr(\ccal)$-enriched category associated to the closed symmetric monoidal structure on $\Corr(\ccal)$.
\end{notation}

 
\section{Construction of \texorpdfstring{$n\kr\Sh(X)$}{nSh(X)}}

We now turn to the proof of theorem \ref{theo exists categorification}. We will construct the maps 
\[
n\kr\Sh^\sharp: n\kr\Corr(\ccal) \rightarrow n\kr\Pr_{\Sh}
\] as functors of pointed $n$-categories. The compatibilities between different values of $n$ will be manifest from the construction. Note that the symmetric monoidal structures on $n\kr\Pr_{\Sh}$ and $n\kr\Sh$ follow from these compatibilities.

\begin{notation}\label{notation mu}
Let $\Mcal$ be a symmetric monoidal category. We denote by $\Gamma_\Mcal: \Mcal \rightarrow \operatorname{An}$ the functor corepresented by the unit of $\Mcal$. We equip $\Gamma_\Mcal$ with the induced lax symmetric monoidal structure.
\end{notation}

\begin{remark}
Let $\Mcal$ be a symmetric monoidal category. Then $\Gamma_\Mcal$ is the initial lax symmetric monoidal functor from $\Mcal$ into anima.
\end{remark}

\begin{notation}
Let $n \geq 1$. We let 
\[
n\kr\Mod : n\kr\Cat^{\Pr} \rightarrow (n+1)\kr\Pr
\]
be the canonical functor. In other words, $n\kr\Mod$ is obtained by freely adding colimits of cells of dimension $0 \leq d \leq n$. We note that for each object $\Dcal$ in $n\kr\Cat^{\Pr}$ we have a morphism $Y_\Dcal: \Dcal \rightarrow n\kr\Mod_{\Dcal}$.
\end{notation}

\begin{construction}\label{construction categorification}
Let $\ccal$ be a category with finite limits, and let $\Sh : \Corr(\Ccal) \rightarrow \Pr$ be a lax symmetric monoidal functor. Let $\eta: \Gamma_{\Corr(\Ccal)} \rightarrow \Sh$ be the canonical lax symmetric monoidal natural transformation. For each $n > 1$, we have that $\eta$ induces a functor 
\[
n\kr\Sh^{\sharp, \pre} : n\kr\Corr(\ccal) = (\Gamma_{\Corr(\Ccal)})_! n\kr\Corr^\enr(\ccal) \rightarrow \Sh_! n\kr\Corr^\enr(\Ccal).
\]

We let $n\kr\Pr_{\Sh} = n\kr\Mod_{\Sh_!n\kr\Corr^\enr(\Ccal)}$, and let $n\kr\Sh^\sharp$ be the composite functor
\[
n\kr\Corr(\Ccal) \xrightarrow{n\kr\Sh^{\sharp, \pre}} \Sh_! n\kr\Corr^{\enr}(\Ccal) \xrightarrow{Y_{\Sh_! n\kr\Corr^{\enr}(\Ccal)}} n\kr\Pr_{\Sh}.
\]
\end{construction}

\begin{remark}
Construction \ref{construction categorification} naturally breaks up into several steps. First one has $n\kr\Corr^\enh(\Ccal)$: this is an enrichment of $n\kr\Corr(\Ccal)$ over $\Corr(\Ccal)$, described informally by the requirement that the Hom object between a pair of $(n-1)$-cells corresponding to objects $Y$ and $Z$ in some overcategory $\Ccal_{/S}$ is given by $Y \times_S Z$. 

Then $\Sh_! n\kr\Corr^\enh(\Ccal)$ is defined, which forms the target of the functor $n\kr\Sh^{\sharp,\pre}$. This satisfies all the design criteria for our theorem \ref{theo exists categorification} except for the existence of colimits: every cell in $\Sh_! n\kr\Corr^\enh(\Ccal)$ is geometric. The desired categorifications are finally obtained by adding colimits.
\end{remark}


\section{Descent and affineness}\label{section compute}

We close this note with a discussion of the fundamental tools for working with the categorifications of a sheaf theory. We begin with a discussion of descent.

\begin{definition}
Let $\Ccal$ be a category with finite limits, and let $\Sh: \Corr(\Ccal) \rightarrow \Pr$ be a lax symmetric monoidal functor. Let $f: X \rightarrow Y$ be a map in $\Ccal$ with \v{C}ech nerve $X_\bullet$. We say that $f$ satisfies $\Sh$-codescent if $\Sh(Y)$ is the geometric realization (in $\Pr$) of $\Sh(X_\bullet)$. We say that $f$ satisfies universal $\Sh$-codescent if every base change of $f$ satisfies $\Sh$-codescent.
\end{definition}

\begin{proposition}\label{proposition descent}
Let $\Ccal$ be a category with finite limits, and let $\Sh: \Corr(\Ccal) \rightarrow \Pr$ be a lax symmetric monoidal functor. Let $f: X \rightarrow Y$ be a map in $\Ccal$ with \v{C}ech nerve $X_\bullet$. The following are equivalent:
\begin{enumerate}[\normalfont (1)]
\item $f$ satisfies universal $\Sh$-codescent.
\item $2\kr\Sh(Y)$ is the totalization of $2\kr\Sh(X_\bullet)$.
\item $n \kr \Sh(Y)$ is the totalization of $n\kr\Sh(X_\bullet)$ for all $n \geq 1$. 
\item $n \kr \Sh^\sharp(Y)$ is the totalization of $n\kr\Sh^\sharp(X_\bullet)$ for all $n \geq 1$. 
\end{enumerate}
\end{proposition}
\begin{proof}
The canonical functor $2\kr\Sh(Y) \rightarrow \Tot 2\kr\Sh(X_\bullet)$ admits a fully faithful left adjoint. Consequently, (2) is equivalent to the assertion that the counit of the adjunction is an isomorphism. This can be checked on the generators $\Sh(Y'/Y)$, where $Y' \rightarrow Y$ is a map. We thus see that (2) is equivalent to the assertion that the canonical map
\[
|\Sh(X'_\bullet/Y)| \rightarrow \Sh(Y'/Y)
\]
is an isomorphism, where $X'_\bullet$ denotes the \v{C}ech nerve of the base change $f': X' \rightarrow Y'$ of $f$. The above can be checked by applying Hom from an object $\Sh(Z/Y)$. One thus sees that (2) is equivalent to the assertion that the base change of $f$ to $Z \times_Y Y'$ satisfies $\Sh$-codescent. Since $Z \rightarrow Y $ and $Y' \rightarrow Y$ are arbitrary, we deduce that (2) is equivalent to (1).

Clearly (4) implies (3), which implies (2). Suppose now that (2) holds; note that it also holds for any base change of $f$, given the equivalence with (1). By ambidexterity for pullback and pushforward for $2\kr\Sh$ one sees that for every base change $X' \rightarrow Y'$ of $f$ with \v{C}ech nerve $X'_\bullet$ we have that $2\kr\Sh(Y)$ is the geometric realization of $2\kr\Sh(X'_\bullet)$. Repeating the argument for the equivalence between (1) and (2) one sees that $3\kr\Sh(Y)$ is the totalization of $3\kr\Sh(X_\bullet)$. Arguing inductively, we deduce that $n\kr\Sh(Y)$ is the totalization of $n\kr\Sh(X_\bullet)$ for all $n \geq 2$. The fact that this also holds for $n = 1$ follows by passing to endomorphisms of the unit.

We have now proven that (2) implies (3), so that (1), (2) and (3) are equivalent. It remains to show that these also imply (4). Indeed, to check that $n \kr \Sh^\sharp(Y)$ is the totalization of $n \kr \Sh^\sharp(X_\bullet)$ it suffices to show that this is the case after applying Hom from $n\kr\Sh^\sharp(Z)$ for some $Z$ in $\Ccal$. This reduces to the fact that (3) holds for arbitrary base changes of $f$ (since (3) has been shown to be equivalent to (1), which is stable under base change).
\end{proof}

\begin{remark}
Let $\Ccal$ be a category with finite limits and finite coproducts, and let $\Sh: \Corr(\Ccal) \rightarrow \Pr$ be a lax symmetric monoidal functor. Suppose that for every finite family of objects $X_i$ in $\ccal$ the canonical functor $\Ccal_{/ \amalg X_i} \rightarrow \prod \ccal_{/X_i}$ is an equivalence. Then the following are equivalent:
\begin{enumerate}[\normalfont (1)]
\item $\Sh$ preserves finite coproducts.
\item $n\kr\Sh$ preserves finite coproducts for all $n \geq 1$.
\item $n \kr\Sh^\sharp$ preserves finite coproducts for all $n \geq 1$.
\end{enumerate}
Combined with proposition \ref{proposition descent}, this allows one to obtain concrete criteria for checking if $n\kr\Sh$ satisfies descent with respect to a Grothendieck topology.
\end{remark}

\begin{example}
Suppose that $\Sh = \QCoh$, defined as a sheaf theory on the category of affine schemes. Then the morphisms which admit $\QCoh$-codescent are precisely the covers for the descendable topology of \cite{Mathew}. In particular, for all $n \geq 1$ we have that $n\kr\QCoh$ satisfies descent with respect to this topology.
\end{example}

\begin{example}\label{example indcoh descent}
Suppose that $\Sh = \IndCoh$, defined as a sheaf theory on the category of schemes almost of finite presentation over a field. Then every faithfully flat morphism almost of finite presentation satisfies $\IndCoh$-codescent. It follows that $n\kr\IndCoh$ admits fppf descent. Similarly, one has that $n\kr\IndCoh$ admits descent along proper surjective morphisms.
\end{example}

\begin{example}
Suppose that $\Sh = \QCoh$, defined as a sheaf theory on quasi-compact algebraic stacks with affine diagonal and almost of finite presentation over a field of characteristic zero. Then $n\kr\QCoh$ admits fppf descent thanks to \cite{StefanichDualizability} corollary 5.1.4 (in light of propositions 3.3.5 and 4.1.11).\footnote{We note that in fact the agreement $\QCoh_\co(X) = \QCoh(X)$ is already needed to define $\QCoh$ as a six-functor formalism in this generality.}  Consequently, one sees that $n\kr\QCoh$ is Kan extended from affine schemes. 
\end{example}

We now turn to a discussion of affineness.

\begin{proposition}\label{proposition affineness}
Let $\Ccal$ be a category with finite limits, and let $\Sh: \Corr(\Ccal) \rightarrow \Pr$ be a lax symmetric monoidal functor. Let $n \geq 1$. The following are equivalent:
\begin{enumerate}[\normalfont (1)]
\item The canonical functor $\Mod_{n\kr\Sh(1_\Ccal)} \rightarrow n\kr\Pr_{\Sh}$ is an equivalence.
\item For each pair of objects $X, Y$ of $\Ccal$ we have an equivalence
\[
n\kr\Sh(X) \otimes_{n\kr\Sh(1_\Ccal)} n\kr\Sh(Y) = n\kr\Sh(X \times Y).
\]
\end{enumerate}
\end{proposition}
\begin{proof}
The fact that (1) implies (2) follows directly from the fact that $n\kr\Sh^\sharp$ is symmetric monoidal. Suppose now that (2) holds. Note that the functor $\Mod_{n\kr\Sh(1_\Ccal)} \rightarrow n\kr\Pr_{\Sh}$ is fully faithful. Consequently, to show (1) it will suffice to show that for every object $X$ in $\Ccal$ the unit map
\[
1_{n\kr\Pr_{\Sh}} \otimes_{n\kr\Sh(1_\Ccal)} n\kr\Sh(X) \rightarrow n\kr\Sh(X)^\sharp
\]
is an isomorphism. To check this it is enough to prove that this is the case after Hom from $n\kr\Sh^\sharp(Y)$ for some $Y$, in which case our assertion reduces to the formula from (2).
\end{proof}

A mild variant of the proof of proposition \ref{proposition affineness} proves the following relative version:

\begin{proposition}\label{prop affineness relativo}
Let $\Ccal$ be a category with finite limits, and let $\Sh: \Corr(\Ccal) \rightarrow \Pr$ be a lax symmetric monoidal functor. Let $X \rightarrow Y$ be a morphism in $\Ccal$, and let $n \geq 1$. The  following are equivalent:
\begin{enumerate}[\normalfont (1)]
\item The canonical functor $\Mod_{n\kr\Sh(X/Y)}((n+1)\kr\Sh(Y)) \rightarrow (n+1)\kr\Sh(X)$ is an equivalence.
\item For every pair of maps $Z \rightarrow X \leftarrow W$ we have an equivalence
\[
n\kr\Sh(Z/Y) \otimes_{n\kr\Sh(X/Y)} n\kr\Sh(W/Y) = n\kr\Sh(Z \times_X W / Y).
\]
\end{enumerate}
\end{proposition}

\begin{corollary}\label{coro kunneth implica}
Let $\Ccal$ be a category with finite limits, and let $\Sh: \Corr(\Ccal) \rightarrow \Pr$ be a lax symmetric monoidal functor. Suppose that for every pair of maps $X \rightarrow S \leftarrow Y$ in $\Ccal$ the canonical functor $\Sh(X) \otimes_{\Sh(S)} \Sh(Y) \rightarrow \Sh(X \times_S Y)$ is an equivalence. Then $(n+1)\kr\Sh(X) = n\kr\Mod_{\Sh(X)}$ for all $n \geq 1$.
\end{corollary}
\begin{proof}
This follows from an inductive application of propositions \ref{proposition affineness} and \ref{prop affineness relativo}.
\end{proof}

\begin{example}
Let $\Sh = \QCoh$, defined on the category of affine schemes. Then it follows from corollary \ref{coro kunneth implica} that the categorification of quasi-coherent sheaves provided by theorem \ref{theo exists categorification} is compatible with the one we constructed in \cite{Thesis}.
\end{example}

\begin{example}
Let $\Sh$ be the Betti sheaf theory, defined on locally compact Hausdorff topological spaces. Then it follows from corollary \ref{coro kunneth implica} that $(n+1)\kr\Sh(X) = n\kr\Mod_{\Sh(X)}$ for all $n \geq 1$.
\end{example}

We finish with the following proposition, which spells out the basic affineness properties of the ind-coherent theory.

\begin{proposition}
Let $\Sh = \IndCoh$, defined on the category of schemes almost of finite presentation over a field $k$.
\begin{enumerate}[\normalfont (1)]
\item Let $X \rightarrow Y$ be a closed immersion. Then we have an equivalence 
\[
2\kr\IndCoh(X) = \Mod_{\IndCoh(X/Y)}(2\kr\IndCoh(Y)).
\]
\item Let $X \rightarrow Y$ be an arbitrary map. Then we have an equivalence
\[
n\kr\IndCoh(X) = \Mod_{(n-1)\kr\IndCoh(X/Y)}(n\kr\IndCoh(Y)).
\]
for all $n \geq 3$.
\item We have an equivalence
\[
n\kr\Pr_{\IndCoh} = \Mod_{(n-1)\kr\Pr_{\IndCoh}}
\]
for all $n \geq 3$.
\item If $k$ is perfect, then $\Pr_{\IndCoh} = \Pr$, and hence the canonical map $2\kr\Pr \rightarrow 2\kr\Pr_{\IndCoh}$ is fully faithful.
\end{enumerate}
\end{proposition}
\begin{proof}
We begin with a proof of (1). The canonical functor
\[
\Mod_{\IndCoh(X/Y)}(2\kr\IndCoh(Y)) \rightarrow 2\kr\IndCoh(X)
\]
is fully faithful, so it will suffice to show that its image generates $2\kr\IndCoh(X)$. Let $Z \rightarrow X$ be a map, and note that $\IndCoh(X \times_Y Z / X)$ is an object in the image of the pullback functor. Since the projection $X\times_Y Z \rightarrow Z$ is a surjective closed immersion, we have that the pushforward map $\IndCoh(X \times_Y Z / X) \rightarrow \IndCoh(Z/X)$ admits a monadic right adjoint. Item (1) now follows from the fact that Eilenberg-Moore objects for monads in presentable higher categories can be computed in terms of Kleisli objects (which are weighted colimits).

We note that (3) follows formally from (2), in light of proposition \ref{proposition affineness}. Item (2) may similarly be reduced to the case $n = 3$. In this case, proposition \ref{proposition affineness} reduces us to showing that for every pair of maps $Z \rightarrow X \leftarrow W$ we have an equivalence
\[
2\kr\IndCoh(Z/Y) \otimes_{2\kr\IndCoh(X/Y)} 2\kr\IndCoh(W/Y) = 2\kr\IndCoh(Z \times_X W / Y).
\]
The left hand side may be rewritten as
\[
2\kr\IndCoh(Z/Y) \otimes 2\kr\IndCoh(W/Y) \otimes_{2\kr\IndCoh(X/Y) \otimes 2\kr\IndCoh(X/Y)} 2\kr\IndCoh(X/Y)
\]
which is equivalent to
\[
2\kr\IndCoh(Z \times_Y W / Y) \otimes_{2\kr\IndCoh(X \times_Y X / Y)} 2\kr\IndCoh(X/Y).
\]
Item (2) will now follows by applying (1) to the closed immersions given by the diagonal $X \rightarrow X \times_Y X$ and its base change to $Z \times_Y W$.

Finally, item (4) is a direct consequence of proposition \ref{proposition affineness}.
\end{proof}


\ifx\inmain\undefined
\bibliographystyle{myamsalpha2}
\bibliography{References}
\fi
 
\end{document}